\documentclass[12pt]{amsart}
	\usepackage{amssymb}
	\usepackage{enumerate}

	\usepackage{tikz}
		\usetikzlibrary{calc}
		\usetikzlibrary{math}
	\usepackage{subcaption}
	\usepackage{graphicx}
	
	\newtheorem{theorem}{Theorem}[section]
	\newtheorem*{blanktheorem}{Theorem}

	\newtheorem{lemma}[theorem]{Lemma}
	\newtheorem{question}[theorem]{Question}

	\newcommand{\bF}{\mathbb F}
	\newcommand{\bK}{\mathbb K}
	\newcommand{\bR}{\mathbb R}

	\newcommand{\cM}{\mathcal{M}}

	\DeclareMathOperator{\cl}{cl}

	\DeclareMathOperator{\poly}{poly}

	\DeclareMathOperator{\rk}{r}

	\newcommand{\del}{\setminus}
	\newcommand{\con}{/}
	
\begin{document}
\sloppy

 \begin{abstract}
We show that each real-representable matroid is a minor of a complex-representable excluded minor for real-representability.
More generally, for an infinite field $\bF_1$ and a field extension $\bF_2$, 
if $\bF_1$-representability is not equivalent to $\bF_2$-representability, then each $\bF_1$-representable matroid is a
minor of a $\bF_2$-representable excluded minor for $\bF_1$-representability.
 \end{abstract}

\title[On the excluded minors for real-representability]{On the complex-representable excluded minors for real-representability}
\author[Campbell, Geelen]{Rutger Campbell, Jim Geelen}
\address{Department of Combinatorics and Optimization,
	University of Waterloo, Canada}
\thanks{This research was partially supported by grants from the
Office of Naval Research [N00014-10-1-0851] and NSERC [203110-2011] as
well as a NSERC Alexander Graham Bell Canada Graduate Scholarship-Doctoral Program [PGSD3-489418-2016].}

\date{\today}
\maketitle

\section{Introduction}

Consider the problem of characterizing the set of excluded minors for the class of real-representable matroids.
For many classes, the excluded minors provide a concise characterization; for instance, 
Tutte~[\ref{Tuttebinary}] showed that a matroid is binary precisely when it does not contain a $U_{2,4}$-minor.
In contrast to this, Lazarson~[\ref{Lazarson}, Theorem 1] showed that there are infinitely many excluded minors
for real-representability. This in itself
does not preclude the possibility of a simple structural description.
For example, Bonin~[\ref{Lattice}, Theorem 3.1] described the excluded minors for
lattice path matroids, despite the fact that the list is infinite.

However, Mayhew, Newman, and Whittle~[\ref{MNW}] have effectively settled the
matter by proving the following striking result.
\begin{theorem}\label{MNWr}
Each real-representable matroid is a minor of an excluded minor for real-representability.
\end{theorem}
This essentially says that the excluded minors are at least as structurally complex
as the real-representable matroids themselves.
In hindsight this result may not seem so surprising when we consider that matroids are rather wild in comparison
with representable matroids; for example, the number of $n$-element matroids is
$2^{\frac{1}{\poly(n)}2^n}$ whereas the number of representable $n$-element 
matroids is only $2^{\poly(n)}$, where in each case $\poly(n)$ denotes a function that
is bounded above and below 
by a polynomial; see~[\ref{Knuth}] and~[\ref{Nelson}] respectively. 
Surprisingly however,
the same issue arises even if we only try to describe the representable excluded minors for real-representability:
we prove that even the complex-representable excluded minors are at least as wild
as the class of real-representable matroids.
\begin{theorem}\label{mainr}
Each real-representable matroid is a minor of a complex-representable excluded minor for real-representability.
\end{theorem}

More generally, one might consider characterizing the excluded minors for real representability within any
given minor closed class $\cM$. Of
particular interest is when $\cM$
contains all real-representable matroids, since 
this would characterize real-representability.
We  extend Theorem~\ref{mainr} by proving that
the excluded-minors within 
any ``natural'' proper superset $\cM$,
are at least as wild as real-representable matroids themselves.

Here we say that a class $\cM$ is \emph{natural},
when it is closed under isomorphism, minors, adding coloops,
direct sums, and ``principal extensions" (which are defined in Section~\ref{SUITABLE}).
Note that the class of real representable matroids is natural,
so it is reasonable to consider supersets that are also closed under these operations.

We prove the following generalization of Theorem~\ref{mainr}.
\begin{theorem}\label{main}
If $\bF$ is an infinite field and 
$\cM$ is a natural class that properly contains all $\bF$-representable matroids,
then each $\bF$-representable matroid is a minor of an excluded minor for $\bF$-representability
that is contained in $\cM$.
\end{theorem}

The original result of Mayhew, Newman, and Whittle 
also applies to any infinite field~[\ref{MNW}].
Mat\'{u}\v{s} has further generalized their result
to other classes 
such as the class of matroids that are algebraic over a given field,
and the class of almost entropic matroids
~[\ref{Matus}].

\section{Natural Classes}\label{SUITABLE}

Let $F$ be a flat of a matroid $M$.
A \emph{principal extension of $M$ into the flat $F$}
is the matroid $M'$ on a ground set $E(M)\cup\{e\}$
where $M'\del e=M$ and
a subset of $E(M)$ spans $e$
if and only if it spans $F$.
We say that $M'$ is obtained by \emph{freely placing $e$ in $F$}, and \emph{freely placing $e$} when $F=E(M)$.

Recall that a class $\cM$ is \emph{natural}
when it is closed under isomorphism, minors, adding coloops,
principal extensions, and direct sums.

Note that the class of all matroids is natural.
We now give some interesting natural matroid classes.

\subsection*{Matroids representable over a fixed infinite field}
For an infinite field $\bF$, it is well-known and easy to show that the class of 
$\bF$-representable matroids is closed under isomorphisms, minors, adding coloops,
principal extensions, and direct sums;
see~[\ref{MNW},~Lemma~2.1] for principal extensions.

\subsection*{Gammoids}
As we will see in Theorem~\ref{gammoids},
the class of {\em gammoids} (see Oxley~[\ref{O},~page 97~and~109] for a definition)
is the minimal class that
is closed under isomorphisms, minors, adding coloops, and principal extensions.
It is also easy to show that the class of gammoids is closed under direct sums.
Combining these results gives us that the class of gammoids is the minimal natural class.

\subsection*{Orientable matroids}
The class of {\em oriented matroids} (see Oxley~[\ref{O},~page~401] for a definition)
is natural; this follows from results in the book~[\ref{Orientable},~page~330]
by Bj\"{o}rner et al.

\subsection*{Algebraic matroids for a fixed field}
Let $\bF$ be a field.
A matroid $M$ is \emph{algebraic over $\bF$} when there exists a field extension $\bK$ of $\bF$ 
and a map $\phi$ assigning each element $e$ in $E(M)$ to an element $\phi(e)$ of $\bK$ such that $r_M(X)=\dim_{tr}(\bF (\phi(X))$ for each $X\subseteq E(M)$,
where $\dim_{tr}$ is the transcendence dimension over $\bF$. We call $\phi$ an \emph{algebraic representation over $\bF$}.
The class of matroids algebraic over a field $\bF$, is
closed under minors and principal extensions;
see~[\ref{O},~Corollary~6.7.14] and~[\ref{Matus},~Lemma~13] respectively.
The class of algebraic matroids is also closed under direct sums
(and, hence, also under adding coloops)
since we can declare variables from different transcendental extension
fields to be algebraically independent.

Furthermore, it follows directly from the definition
that the intersections of natural classes are also natural.
Thus, under the subset relation, natural classes of matroids form a lattice.

However, there are certainly interesting classes that are not \emph{natural}.
For a prime power $q$, the $\mbox{GF}(q)$-representable matroids are not a natural class:
the uniform matroid $U_{2,q+1}$ is $\mbox{GF}(q)$-representable while the principal extension $U_{2,q+2}$ is not.
More generally, as the class of gammoids is the minimal natural class,
any class of matroids that does not contain all gammoids is not natural,
regardless of how basic it is.

\section{Preliminaries}\label{PROPERTIES}

We use terminology and notation from Oxley~[\ref{O}].

\subsection*{$N$-constructed matroids}

For a matroid $N$, we say that a matroid $M$ is {\em $N$-constructed} if it can be obtained from $N$ by a sequence
of the following operations: relabelling elements, deletion, contraction, adding coloops, and principal extensions.

Consequently, if $N$ is a matroid in a natural class $\cM$ and $M$ is an $N$-constructed matroid,
then $M$ is contained in $\cM$.

Let $e$ be an element of a matroid $M$. Recall that the
{\em series extension} of $e$ in $M$ is the matroid $M'$
obtained by coextending $M$ by an element $e'$ so that $\{e,e'\}$ is a
series pair.
\begin{lemma}\label{series}
Let $e$ be an element of a matroid $M$. If $M'$ is the series extension of $e$
in $M$, then $M'$ is $M$-constructed.
\end{lemma}

\begin{proof}
Let $M_1$ be obtained from $M$ by adding a coloop $e'$ and then freely placing an
element $e''$ in the flat spanned by $\{e,e'\}$. Then $M'$ is obtained from $M_1\del e$ 
by relabelling $e''$ as $e$.
\end{proof}

The following result is essentially due to Mayhew, Newman, and Whittle~[\ref{MNW},~Lemma~2.2].
\begin{lemma}\label{constructed}
For any matroid $N$, there is an $N$-constructed matroid $N'$ such that $N'$ has an $N$-minor
and that the ground-set of $N'$ can be partitioned into two bases.
\end{lemma}

\begin{proof}
Let $A_0$ and $B_0$ denote two $r(N)$-element sets that are disjoint from
$E(N)$ and from each other. We extend $N$ by adding the elements
$A_0\cup B_0$ freely to obtain the matroid $N_1$; thus $A_0$ and 
$B_0$ are bases of $N_1$.
Next, we construct $N'$ from $N_1$ by a sequence of series extensions 
for each element in $E(N)$; we relabel the elements so that, for each $e\in E(N)$,
the corresponding series-pair in $N'$ is $\{e_1,e_2\}$.
Note $N'$ has bases
$A_1=A_0\cup\{e_1:e\in E(N)\}$ and
$B_1=B_0\cup\{e_2:e\in E(N)\}$ which partition $E(N')$, as required.
\end{proof}

\subsection*{Gammoids}

By combining known results,
we show that gammoids are the minimal class that
is closed under isomorphisms, minors, adding coloops, and principal extensions.
We can restate this as the follows.
\begin{theorem}\label{gammoids}
A matroid is a gammoid if and only if it is $U_{0,0}$-constructible.
\end{theorem}
\begin{proof}
Brylawski showed that a {\em transversal matroid} (see Oxley~[\ref{O},~page~45] for a definition)
can be considered as the affine matroid of a collection of points $S$ in $\bR^{r-1}$ which are freely placed on the faces (of possibly varying dimension) of a simplex in $\bR^{r-1}$~[\ref{Brylawski},~Theorem~3.1~and~Corollary~3.1] (see also Oxley~[\ref{O},~Proposition~11.2.26]).
As a immediate consequence, all transversal matroids are $U_{0,0}$-constructible.
Next, Ingleton and Piff showed that the class of gammoids
 is the class of transversal matroids closed under contraction~[\ref{Gammoids}, Theorem 3.5].
Thus all gammoids are $U_{0,0}$-constructible.

Another consequence of the above description is that gammoids are closed under isomorphisms, minors, and adding coloops.
It remains to be shown that the class of gammoids is closed under principal extensions.
For a gammoid $G$ with ground set $E$, say $G$ is the contraction of a transversal matroid $T$ by a set $X\subseteq E(M)$.
Say that $T$ is the affine matroid of points $S$ in $\bR^{r-1}$ which are freely place on the faces of a simplex with vertex set $B\subseteq\bR^{r-1}$.
Suppose we have $e\in E$ that does not lie on a vertex of this simplex, but lies in the affine span of $B_e\subseteq B$.
We coextend $T$, by a turning $e$ into a series pair $\{x,e\}$ to get $T'$. 
By embedding $\bR^{r-1}$ in $\bR^r$, we have that $T'$ is an affine matroid of points that lie in the faces of a simplex with vertex set $B\cup\{e\}$; the points in $S-\{e\}$ lie in the span of $B$ as before, but $x$ lies in the face spanned by $B_e\cup\{e\}$.
 Note $T'\con X\cup\{x\}=T\con X=G$ but now $e$ is on a vertex of the simplex in the affine representation. 
 In this way we may assume that all of $E$ lies on the vertices of a affine representation of $T$. 
 Consider a principal extension $G_F$ into a flat $F\subseteq E$. 
 Note that the principal extension $T_F$ into $F\subseteq E$ is also a transversal matroid as $F$ are the vertices of a face. 
 Thus $T_F\con X=G_F$ is a gammoid, as we wanted to show.
\end{proof}

In a matroid $M$, we say an element $p$ is {\em freer} than an element $q$,
when every subset of $E(M)\setminus\{p,q\}$ that spans $p$ also spans $q$.
A pair of elements $(p,q)$ is {\em incomparable}
when there is a set that spans $p$ but not $q$
and a set that spans $q$ but not $p$.

\begin{lemma}\label{incomp}
	Matroids with no incomparable pair are gammoids. 
\end{lemma}
\begin{proof}
         Suppose that $M$ has no incomparable pair. Then there is an
         ordering $(e_1,\ldots,e_n)$ of $E(M)$ such that $e_j$ is freer than $e_i$ whenever
         $1\le i<j\le n$. 
         So either $M$ is the empty matroid and hence $U_{0,0}$-constructible,
         or else $M$ has a freest element $e_n$.
         Now either $e_n$ is a coloop in $M$ or $M$ is obtained by placing $e_n$ freely in $M\del e_n$.
         Note that $M\del e_n$ has no incomparable pair so we may inductively assume that $M\del e_n$ is
         $U_{0,0}$-constructible, and, hence, $M$ is $U_{0,0}$-constructible.
\end{proof}

Our proof of Theorem~\ref{main} is based on a construction that starts with an excluded-minor
contained in $\cM$. However, we cannot take 
an arbitrary excluded minor; we require a ``special" pair of elements that are provided by the lemma 
below.

If $M_1$ and $M_2$ are matroids on a common ground set $E$, then we say that 
$M_2$ is {\em freer} than $M_1$ if $r_{M_2}(X)\ge r_{M_1}(X)$ for each subset $X$ of $E$.
Let $a$ and $b$ be distinct elements of a matroid $M$ and let $M'$ denote the matroid obtained
from $M$ by freely adding a new element $c$ into the flat spanned by $\{a,b\}$.
We denote by $M_{a\rightarrow b}$ the matroid obtained from $M'\del a$ by relabelling $c$ as $a$.
Note that $c$ is freer than $a$ in $M'$ and hence $M_{a\rightarrow b}$ is 
freer than $M$.
\begin{lemma}\label{special-pair}
Let $\cM_1$ and $\cM_2$ be natural classes of matroids. If $\cM_1\subsetneq \cM_2$, then
there is an excluded-minor $L$ for $\cM_1$ in $\cM_2$ with 
a pair $\{p,q\}$ of incomparable elements  such that $L_{p\rightarrow q}$ and
$L_{q\rightarrow p}$ are both contained in $\cM_1$.
\end{lemma}

\begin{proof}
Since $\cM_1\subsetneq \cM_2$, there is an excluded-minor for 
$\cM_1$ in $\cM_2$.
Among all excluded-minors for $\cM_1$ in $\cM_2$ we choose $L$ satisfying:
\begin{itemize}
\item $|L|$ is minimum, and
\item subject to this, $L$ is {\em freest} with ground set $E(L)$ (that is, there is no
other excluded minor $L'$ with ground set $E(L)$ that is freer than $L$).
\end{itemize}
By Theorem~\ref{gammoids}, $\cM_1$ contains all gammoids and, by Lemma~\ref{incomp},
$L$ has an incomparable pair $\{p,q\}$. Note that $L_{p\rightarrow q}$ and
$L_{q\rightarrow p}$ are both $L$-constructed and hence they are both contained
in $\cM_2$. Moreover,
both $L_{p\rightarrow q}$ and $L_{q\rightarrow p}$ are freer than $L$.
 So, by our choice of $L$, both $L_{p\rightarrow q}$ and
$L_{q\rightarrow p}$ are contained in $\cM_1$, as required.
\end{proof}

\subsection*{Extending into a $3$-separation}

The \emph{local connectivity}
between two sets $S$ and $T$ of a matroid $M$ is
$$\sqcap_M(S,T)=\rk_M(S)+\rk_M(T)-\rk_M(S\cup T).$$
Now, for disjoint sets $X$, $Y$, and $C$ in $M$, we have 
$$ \sqcap_{M\con C}(X,Y) = \sqcap_M(X,Y\cup C) - \sqcap_M(X,C),$$
which can be easily verified by expanding both sides. We will also use the
fact that, if $\sqcap_M(X,Y)=0$ and $e$ is spanned by both $X$ and $Y$, then
$e$ is a loop; this follows since $r_M(\{e\})\le r_M(\cl_M(X)\cap\cl_M(Y))
\le r_M(\cl_M(X)) +r_M(\cl_M(Y)) -r_M(\cl_M(X)\cup\cl_M(Y))
=r_M(X)+r_M(Y)-r_M(X\cup Y) = \sqcap_M(X,Y)=0$.

Let $(S_1,S_2)$ be a $3$-separation in a matroid $M'$ and let $M$ 
be obtained from $M'$ by extending by a non-loop element $e$ 
into the closures of both $S_1$ and $S_2$. Unlike the case with $2$-separations,
this does not uniquely determine $M$. However, under 
some additional hypotheses, the following result shows that we can uniquely determine $M$.
\begin{lemma}\label{unique}
Let $e$ be a non-loop element of a matroid $M$, let $(S_1,S_2)$ be a $3$-separation
of $M\del e$, and let $Y_1\subseteq S_1$ and $Y_2\subseteq S_2$ such that
$\sqcap_M(Y_1,S_2) = 1$, $\sqcap_M(S_1,Y_2)=1$, and $e$ is spanned by both
$Y_1$ and $Y_2$ in $M$. Then a flat $F$ of $M$ spans $e$ if and only if either
	\begin{enumerate}[(i)]
		\item $\sqcap_M(F\cap S_1,Y_2)=1$ or $\sqcap_M(Y_1,F\cap S_2)=1$, or
		\item $\sqcap_M(F\cap S_1,S_2)=\sqcap_M(S_1,F\cap S_2)=1$ and $\sqcap_M(F\cap S_1,F\cap S_2)=0$.
	\end{enumerate}
\end{lemma}
\begin{proof}
Let $F_1=F\cap S_1$ and $F_2=F\cap S_2$.
First, suppose that $\sqcap_M(F_1,Y_2)=1$.
Then $\sqcap_{M\con F_1}(S_1-F_1,Y_2) =\sqcap_M(S_1,Y_2) - \sqcap_M(F_1,Y_2)=0$.
However, $e$ is in the closure of both $S_1-F_1$ and $Y_2$ in $M\con F_1$. Thus $e$ is a loop 
in $M\con F_1$ and hence $e$ is spanned by $F$.
By symmetry, if $\sqcap_M(Y_1,F_2)=1$, then $e$ is spanned by $F$.

Now suppose that $\sqcap_M(F_1,S_2)=\sqcap_M(S_1,F_2)=1$ and $\sqcap_M(F_1,F_2)=0$.
Then
\begin{eqnarray*}
\sqcap_{M\con F}(S_1-F_1,S_2-F_2) &=&
\sqcap_{M\con F_1}(S_1-F_1,S_2) - \sqcap_{M\con F_1}(S_1-F_1,F_2)\\
&=& \sqcap_M(S_1,S_2) -\sqcap_M(F_1,S_2)\\ &&-\sqcap_M(S_1,F_2) + \sqcap_M(F_1,F_2)\\
&=& 0.
\end{eqnarray*}
However, $e$ is spanned by both $S_1-F_1$ and $S_2-F_2$ in $M\con F$.
Thus $e$ is a loop in $M\con F$ and, hence, $F$ spans $e$.

Conversely, suppose that $F$ spans $e$ and hence that $e$ is a loop in $M\con F$.
We may assume that $e$ is not spanned by either $F_1$ or $F_2$
since otherwise $(i)$ holds. Since $e$ is spanned by $F_2$ in $M\con F_1$,
we have $\sqcap_M(F_1,S_2)=1$. Similarly $\sqcap_M(S_1,F_2)=1$.
Moreover, again since $e$ is spanned by $F_2$ in $M\con F_1$, we have
$1=\sqcap_{M\con F_1}(S_1-F_1,F_2) = \sqcap_{M}(S_1,F_2) - \sqcap_M(F_1,F_2) = 1- \sqcap_M(F_1,F_2)$
and, hence $\sqcap_M(F_1,F_2)=0$, so $(ii)$ holds.
\end{proof}

\section{The Construction}\label{CONSTRUCTION}

In the construction and in all subsequent results in this section,
\begin{itemize}
\item $L$ and $N$ are matroids with disjoint ground sets,
\item $(A,B)$ is a partition of $E(N)$ into two bases, and
\item $p$ and $q$ are distinct elements of $L$.
\end{itemize}
We build an $(N\oplus L)$-constructed matroid 
$M(L,p,q;N, A,B)$ as follows.
The input and output of the construction process are depicted in Figure~\ref{Fconstruction}.

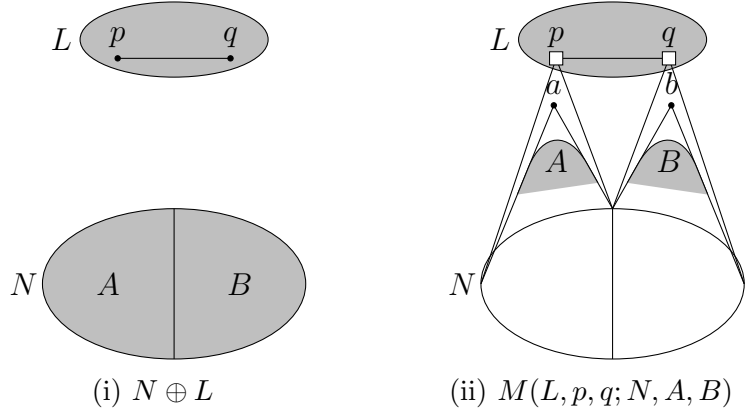
\begin{figure}
    \centering
	\begin{subfigure}[h]{0.45\textwidth}
	    \centering
		\begin{tikzpicture}[scale=0.5]
			\def\S{2}
			\coordinate [label=above:$p$] (P) at (3,8);
			\coordinate [label=above:$q$] (Q) at (6,8);
			\coordinate (N_1) at (1,2);
			\coordinate (N_2) at (8,2);
			\coordinate (N_3) at ($ 0.5*(N_1)+0.5*(N_2)+(0,\S) $);
			\filldraw[fill=lightgray] ($ 0.5*(P)+0.5*(Q)+(0,.5) $) circle [x radius=2.5, y radius=1];
			\coordinate [label=above:$p$] (p) at (P);
			\coordinate [label=above:$q$] (q) at (Q);
			\filldraw[fill=lightgray] ($ 0.5*(N_1)+0.5*(N_2) $) circle [x radius=3.5, y radius=\S];
			\draw ($ 0.5*(N_1)+0.5*(N_2)+(0,\S) $) -- ($ 0.5*(N_1)+0.5*(N_2)-(0,\S) $);
			\node (L) at ($ (p)+(-1.5,0.5) $){$L$};
			\node (N) at ($ (N_1)-(.5,0) $){$N$};

			\filldraw (p) circle (2pt);
			\filldraw (q) circle (2pt);

			\draw (p) -- (q);
			\node (A) at ($ 0.75*(N_1)+0.25*(N_2) $){$A$};
			\node (B) at ($ 0.25*(N_1)+0.75*(N_2) $){$B$};
		\end{tikzpicture}
		\caption{$N\oplus L$}
	\end{subfigure}
	\begin{subfigure}[h]{0.45\textwidth}
	    \centering
		\begin{tikzpicture}[scale=0.5]
			\def\S{2}
			\coordinate [label=above:$p$] (P) at (3,8);
			\coordinate [label=above:$q$] (Q) at (6,8);
			\coordinate (N_1) at (1,2);
			\coordinate (N_2) at (8,2);
			\coordinate (N_3) at ($ 0.5*(N_1)+0.5*(N_2)+(0,\S) $);
			\filldraw[fill=lightgray] ($ 0.5*(P)+0.5*(Q)+(0,.5) $) circle [x radius=2.5, y radius=1];
			\coordinate [label=above:$p$] (p) at (P);
			\coordinate [label=above:$q$] (q) at (Q);
			\draw ($ 0.5*(N_1)+0.5*(N_2) $) circle [x radius=3.5, y radius=\S];
			\draw ($ 0.5*(N_1)+0.5*(N_2)+(0,\S) $) -- ($ 0.5*(N_1)+0.5*(N_2)-(0,\S) $);
			\node (L) at ($ (p)+(-1.5,0.5) $){$L$};
			\node (N) at ($ (N_1)-(.5,0) $){$N$};
			\coordinate [label=above:$a$] (a) at ($ 0.125*(N_1)+0.75*(p)+0.125*(N_3) $);
			\coordinate [label=above:$b$] (b) at ($ 0.125*(N_3)+0.75*(q)+0.125*(N_2) $);

			\filldraw (a) circle (2pt);
			\filldraw (b) circle (2pt);
			\draw (p) -- (q);

			\draw (N_1) -- (p) -- (N_3) -- (q) -- (N_2);
			\draw (N_1) -- (a) -- (N_3) -- (b) -- (N_2);
			\filldraw[fill=lightgray,rounded corners=25pt] ($ 0.5*(N_1)+ 0.5*(a) $) -- (a) -- ($ 0.75*(N_3)+ 0.25*(a) $);
			\filldraw[fill=lightgray,rounded corners=25pt] ($ 0.75*(N_3)+ 0.25*(b) $) -- (b) -- ($ 0.5*(N_2)+ 0.5*(b) $);
			\node (A) at ($ 0.58*(a)+0.17*(N_1)+0.25*(N_3) $){$A$};
			\node (B) at ($ 0.58*(b)+0.17*(N_2)+0.25*(N_3) $){$B$};

			\filldraw[fill=white] ($ (p)+(-5pt,-5pt) $) -- ($ (p)+(-5pt,5pt) $) -- ($ (p)+(5pt,5pt) $) -- ($ (p)+(5pt,-5pt) $) -- cycle;
			\filldraw[fill=white] ($ (q)+(-5pt,-5pt) $) -- ($ (q)+(-5pt,5pt) $) -- ($ (q)+(5pt,5pt) $) -- ($ (q)+(5pt,-5pt) $) -- cycle;
		\end{tikzpicture}
		\caption{$M(L,p,q;N,A,B)$}
	\end{subfigure}
	\caption{The input and output of the construction process.}\label{Fconstruction}
\end{figure}

We first build a matroid $M_1(L,p,q;N,A,B)$ from
$N\oplus L$ by freely placing elements $a$ and $b$ 
in the flats spanned by $E(N)\cup \{p\}$ and $E(N)\cup \{q\}$ respectively; then,
for each $x\in A$, we freely place an element $x_{a}$ in $\{x,a\}$,
and, for each $y\in B$, we freely place an element $y_{b}$ in $\{y,b\}$.

We then obtain $M_2(L,p,q;N, A,B)$ from $M_1(L,p,q;N,A,B)$ as follows:
for each $x\in A$ and $y\in B$, we delete
$x$ and $y$ and relabel $x_{a}$ and $y_{b}$ as $x$ and $y$ respectively.
Finally, we let $M(L,p,q;N, A,B) = M_2(L,p,q;N, A,B)\del \{p,q\}$.
Note that if $\{p,q\}$ is an independent pair in $L$,
then $M(L,p,q;N,A,B)$ contains $N$ as the minor
$M(L,p,q;N,A,B)\con\{a,b\}\del E(L)$.

The main result in this section is the following.
\begin{theorem}\label{construction}
Let $\bF$ be an infinite field. If
\begin{itemize}
\item[(i)] $L$ is an excluded-minor for the class of $\bF$-representable matroids,
\item[(ii)] $p$ and $q$ are an incomparable pair of elements in $L$ such that $L_{p\rightarrow q}$ and 
$L_{q\rightarrow p}$ are both $\bF$-representable, 
\item[(iii)] $N$ is an $\bF$-representable matroid, and
\item[(iv)] $(A,B)$ is a partition of $E(N)$ into bases,
\end{itemize}
then $M(L,p,q;N,A,B)$ is an excluded-minor for the class of $\bF$-representable matroids.
\end{theorem}

We will start by proving that $M(L,p,q;N,A,B)$ is not $\bF$-representable. For this we require the following results.
The first of these results gives us some simple structural properties of $M(L,p,q;N,A,B)$.
\begin{lemma}\label{separation}
Let $M=M(L,p,q;N,A,B)$.
If $\{p,q\}$ is independent and coindependent in $L$, then $(E(L)-\{p,q\}, E(N)\cup\{a,b\})$ is a $3$-separation in $M$ and
$\sqcap_{M}(A\cup\{a\},E(L)-\{p,q\})=\sqcap_{M}(B\cup\{b\},E(L)-\{p,q\})=1$.
\end{lemma}

\begin{proof}
Let $M=M(L,p,q;N,A,B)$ and $M_2=M_2(L,p,q;N,A,B)$. 
Note that $\sqcap_{M_2}(E(L), E(N)\cup\{a,b\})=2$ and
$\sqcap_{M_2}(A\cup\{a\},E(L))=\sqcap_{M_2}(B\cup\{b\},E(L))=1$.
Then, since $\{p,q\}$ is coindependent in $L$, we have that
$(E(L)-\{p,q\}, E(N)\cup\{a,b\})$ is a $3$-separation in $M$ and
$\sqcap_{M}(A\cup\{a\},E(L)-\{p,q\})=\sqcap_{M}(B\cup\{b\},E(L)-\{p,q\})=1$.
\end{proof}

The following result shows that, if we extend $M(L,p,q;N,A,B)$ by nonloop elements $p$ and $q$ such that
$p$ is spanned by both $A\cup\{a\}$ and $E(L)-\{p,q\}$ whereas
$q$ is spanned by both $B\cup\{b\}$ and $E(L)-\{p,q\}$, then we retrieve $M_2(L,p,q;N,A,B)$.
\begin{lemma}\label{retrieval}
Let $M'$ be an extension of $M(L,p,q;N,A,B)$ by nonloop elements $p$ and $q$ such that
$p$ is spanned by both $A\cup\{a\}$ and $E(L)-\{p,q\}$ whereas
$q$ is spanned by both $B\cup\{b\}$ and $E(L)-\{p,q\}$. If $\{p,q\}$ is an incomparable pair
in $L$, then $M'=M_2(L,p,q;N,A,B)$.
\end{lemma}

\begin{proof}
Let $M=M(L,p,q;N,A,B)$ and $M_2=M_2(L,p,q;N,A,B)$. 
Since $\{p,q\}$ is an incomparable pair in $L$, $\{p,q\}$ is both
independent and coindependent and there exist sets
$Y_p,Y_q\subseteq E(L)-\{p,q\}$ such that
$Y_p$ spans $p$ but not $q$ and $Y_q$ spans $q$ but not $p$.
Note that $\sqcap_M(Y_p,E(N)\cup\{a,b\}) = \sqcap_M(Y_p,A\cup \{a\}) = 1$ and
$\sqcap_M(Y_q,E(N)\cup\{a,b\}) = \sqcap_M(Y_q,A\cup \{a\}) = 1$.
Moreover $M'\del\{p,q\}=M_2\del\{p,q\}=M$. It follows from Lemma~\ref{unique} that 
$M'=M_2$.
\end{proof}

We can now show that $M(L,p,q;N,A,B)$ is not $\bF$-representable.
\begin{lemma}\label{CONii}
If $L$ is not representable over a field $\bF$ and $\{p,q\}$ is an incomparable pair in $L$, then
$M(L,p,q;N,A,B)$ is not $\bF$-representable either.
\end{lemma}

\begin{proof}
Let $M=M(L,p,q;N,A,B)$ and $M_2=M_2(L,p,q;N,A,B)$. 
Since $\{p,q\}$ is an incomparable pair in $L$, $\{p,q\}$ is both
independent and coindependent.
We may assume, towards a contradiction, that $M$ is $\bF$-representable.
Then, by Lemma~\ref{separation}, there is an $\bF$-representable extension 
$M'$ of $M$ by nonloop elements $p$ and $q$ such that
$p$ is spanned by both $A\cup\{a\}$ and $E(L)-\{p,q\}$ whereas
$q$ is spanned by both $B\cup\{b\}$ and $E(L)-\{p,q\}$.
By Lemma~\ref{retrieval}, we have $M'=M_2$.
However, $M'|E(L) = M_2|E(L) = L$, which is not $\bF$ representable.
This contradiction completes the proof.
\end{proof}

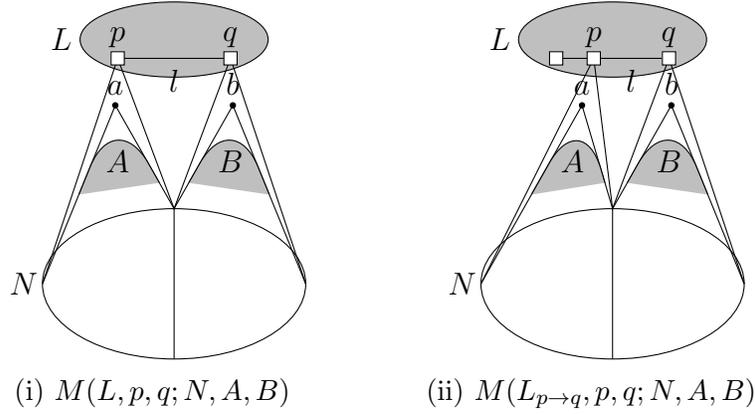
\begin{figure}
    \centering
	\begin{subfigure}[h]{0.45\textwidth}
	    \centering
		\begin{tikzpicture}[scale=0.5]
			\def\S{2}
			\coordinate [label=above:$p$] (P) at (3,8);
			\coordinate [label=above:$q$] (Q) at (6,8);
			\coordinate (N_1) at (1,2);
			\coordinate (N_2) at (8,2);
			\coordinate (N_3) at ($ 0.5*(N_1)+0.5*(N_2)+(0,\S) $);
			\filldraw[fill=lightgray] ($ 0.5*(P)+0.5*(Q)+(0,.5) $) circle [x radius=2.5, y radius=1];
			\coordinate [label=above:$p$] (p) at (P);
			\coordinate [label=above:$q$] (q) at (Q);
			\draw ($ 0.5*(N_1)+0.5*(N_2) $) circle [x radius=3.5, y radius=\S];
			\draw ($ 0.5*(N_1)+0.5*(N_2)+(0,\S) $) -- ($ 0.5*(N_1)+0.5*(N_2)-(0,\S) $);
			\node (L) at ($ (p)+(-1.5,0.5) $){$L$};
			\node (N) at ($ (N_1)-(.5,0) $){$N$};
			\coordinate [label=above:$a$] (a) at ($ 0.125*(N_1)+0.75*(p)+0.125*(N_3) $);
			\coordinate [label=above:$b$] (b) at ($ 0.125*(N_3)+0.75*(q)+0.125*(N_2) $);

			\filldraw (a) circle (2pt);
			\filldraw (b) circle (2pt);
			\draw (p) -- (q);
			\coordinate [label=below:$l$] (l) at ($ 0.5*(p)+0.5*(q) $);

			\draw (N_1) -- (p) -- (N_3) -- (q) -- (N_2);
			\draw (N_1) -- (a) -- (N_3) -- (b) -- (N_2);
			\filldraw[fill=lightgray,rounded corners=25pt] ($ 0.5*(N_1)+ 0.5*(a) $) -- (a) -- ($ 0.75*(N_3)+ 0.25*(a) $);
			\filldraw[fill=lightgray,rounded corners=25pt] ($ 0.75*(N_3)+ 0.25*(b) $) -- (b) -- ($ 0.5*(N_2)+ 0.5*(b) $);
			\node (A) at ($ 0.58*(a)+0.17*(N_1)+0.25*(N_3) $){$A$};
			\node (B) at ($ 0.58*(b)+0.17*(N_2)+0.25*(N_3) $){$B$};

			\filldraw[fill=white] ($ (p)+(-5pt,-5pt) $) -- ($ (p)+(-5pt,5pt) $) -- ($ (p)+(5pt,5pt) $) -- ($ (p)+(5pt,-5pt) $) -- cycle;
			\filldraw[fill=white] ($ (q)+(-5pt,-5pt) $) -- ($ (q)+(-5pt,5pt) $) -- ($ (q)+(5pt,5pt) $) -- ($ (q)+(5pt,-5pt) $) -- cycle;
		\end{tikzpicture}
		\caption{$M(L,p,q;N,A,B)$}
	\end{subfigure}
	\begin{subfigure}[h]{0.45	\textwidth}
	    \centering
		\begin{tikzpicture}[scale=0.5]
			\def\S{2}
			\coordinate [label=above:$p$] (P) at (3,8);
			\coordinate [label=above:$q$] (Q) at (6,8);
			\coordinate (N_1) at (1,2);
			\coordinate (N_2) at (8,2);
			\coordinate (N_3) at ($ 0.5*(N_1)+0.5*(N_2)+(0,\S) $);
			\filldraw[fill=lightgray] ($ 0.5*(P)+0.5*(Q)+(0,.5) $) circle [x radius=2.5, y radius=1];
			\coordinate [label=above:$p$] (p) at (4,8);
			\coordinate [label=above:$q$] (q) at (Q);
			\draw ($ 0.5*(N_1)+0.5*(N_2) $) circle [x radius=3.5, y radius=\S];
			\draw ($ 0.5*(N_1)+0.5*(N_2)+(0,\S) $) -- ($ 0.5*(N_1)+0.5*(N_2)-(0,\S) $);
			\node (L) at ($ (P)+(-1.5,0.5) $){$L$};
			\node (N) at ($ (N_1)-(.5,0) $){$N$};
			\coordinate [label=above:$a$] (a) at ($ 0.125*(N_1)+0.75*(p)+0.125*(N_3) $);
			\coordinate [label=above:$b$] (b) at ($ 0.125*(N_3)+0.75*(q)+0.125*(N_2) $);

			\filldraw (a) circle (2pt);
			\filldraw (b) circle (2pt);
			\draw (P) -- (Q);
			\coordinate [label=below:$l$] (l) at ($ 0.5*(p)+0.5*(q) $);

			\draw (N_1) -- (p) -- (N_3) -- (q) -- (N_2);
			\draw (N_1) -- (a) -- (N_3) -- (b) -- (N_2);
			\filldraw[fill=lightgray,rounded corners=25pt] ($ 0.5*(N_1)+ 0.5*(a) $) -- (a) -- ($ 0.75*(N_3)+ 0.25*(a) $);
			\filldraw[fill=lightgray,rounded corners=25pt] ($ 0.75*(N_3)+ 0.25*(b) $) -- (b) -- ($ 0.5*(N_2)+ 0.5*(b) $);
			\node (A) at ($ 0.58*(a)+0.17*(N_1)+0.25*(N_3) $){$A$};
			\node (B) at ($ 0.58*(b)+0.17*(N_2)+0.25*(N_3) $){$B$};

			\filldraw[fill=white] ($ (P)+(-5pt,-5pt) $) -- ($ (P)+(-5pt,5pt) $) -- ($ (P)+(5pt,5pt) $) -- ($ (P)+(5pt,-5pt) $) -- cycle;
			\filldraw[fill=white] ($ (p)+(-5pt,-5pt) $) -- ($ (p)+(-5pt,5pt) $) -- ($ (p)+(5pt,5pt) $) -- ($ (p)+(5pt,-5pt) $) -- cycle;
			\filldraw[fill=white] ($ (q)+(-5pt,-5pt) $) -- ($ (q)+(-5pt,5pt) $) -- ($ (q)+(5pt,5pt) $) -- ($ (q)+(5pt,-5pt) $) -- cycle;
		\end{tikzpicture}
		\caption{$M(L_{p\rightarrow q},p,q;N,A,B)$}
	\end{subfigure}
	\caption{Matroids that have the same minor when an element of $A\cup\{a\}$ is deleted
	        or an element of $B\cup\{b\}$ is contracted.}\label{Fcomparison}
\end{figure}

It remains to prove that proper minors of 
$M(L,p,q;N,A,B)$ are $\bF$-representable.
We do this by showing that every proper minor of $M(L,p,q;N,A,B)$
is also a minor of one of $M(L_{p\rightarrow q},p,q;N,A,B)$, $M(L_{q\rightarrow p},p,q;N,A,B)$,
or $M(L',p,q;N,A,B)$, where $L'$ is a proper minor of $L$.
For this we need two preliminary lemmas; the 
first shows that there is a unique set in 
$A\cup B\cup\{a,b\}$ that spans $p$ but not $q$.

\begin{lemma}\label{pnotq}
If $X\subseteq A\cup B\cup\{a,b\}$
spans $p$ but not $q$ in $M_2(L,p,q;N,A,B)$, then $X=A\cup\{a\}$.
\end{lemma}

\begin{proof}
Let $N_2$ denote the restriction of $M_2(L,p,q;N,A,B)$ to
$A\cup B\cup \{a,b,p,q\}$.
We consider an alternate construction of $N_2$.
Let $N_1$ be obtained from $N$ by adding coloops $p$ and $q$ and
adding $a$ freely to the flat spanned by $A\cup\{p\}$ and $b$
freely to the flat spanned by $B\cup\{q\}$; 
then, for each $x\in A$, we freely place $x_{a}$ in $\{x,a\}$
and, for each element $y\in B$, freely place $y_{b}$ in $\{y,b\}$.
Then, we let $N_2$ be obtained by deleting each $x\in A$ and $y\in B$
and relabelling each $x_{a}$ and $y_{b}$ as $x$ and $y$ respectively.
Let $C_1=A\cup\{a,p\}$ and $C_2=B\cup\{b,q\}$.
Note that $C_1$ is a circuit of $N_1$ and hence also in $N_2$. Moreover, since
each of the elements of $B$ has been ``lifted" towards $b$, the set $C_1$ is also a
hyperplane of $N_2$.

Note that, since $C_1$ is a circuit-hyperplane, it suffices to show
that $C_1$ is the only cyclic flat of $N_2$ that contains $p$ but not $q$.
Suppose that $F\neq C_1$ is a cyclic flat of $N_2$ that contains $p$ but not $q$.
Thus $F\cap C_2\neq \emptyset$. Since $F$ is cyclic and $C_2$ is a cocircuit,
$|C_2\cap F|\ge 2$. Since $q\not\in F$, the flat $F$ contains an element $y\in B$.
Since each element in $B$ is freer than $b$ in $N_2$, we have $b\in F$.
Similarly $a\in F$.
Now $F-\{a,b\}$ is a union of cycles in $N_2\con\{a,b\}$.
However, $N_2\con\{a,b\}= N_0\con\{a,b\}$. Now $a$ is freely placed
in the flat $E(N)\cup\{p\}$  and $\{a,p\}$ is a series-pair
in $N_0$, and hence $p$ is freely placed in $N_0\con \{a,b\}$.
However, $p\in F-\{a,b\}$, and hence $F-\{a,b\}$ contains a basis of $B'$ of $N$.
However, $B'\cup\{a,b\}\subseteq F$ is a basis of $N_2$, contrary to the fact that
$q$ is not contained in the flat $F$.
\end{proof}

The following result captures the difference between the matroids
$M(L,p,q;N,A,B)$ and $M(L_{p\rightarrow q},p,q;N,A,B)$. It will let us show that
when we delete an element in $A\cup\{a\}$ or contract an element in $B\cup\{b\}$ 
we will get the same minor, Figure~\ref{Fcomparison}.
\begin{lemma}\label{difference}
Let $\{p,q\}$ be an incomparable pair in $L$.
Let $X$ be a set of elements in $M(L,p,q;N,A,B)$. If 
$M(L,p,q;N,A,B)$ and $M(L_{p\rightarrow q},p,q;N,A,B)$ differ in rank on $X$, then $X\cap (A\cup B\cup\{a,b\}) = A\cup\{a\}$.
\end{lemma}

\begin{proof}
Let $M=M_2(L,p,q;N,A,B)$ and $M' = M_2(L_{p\rightarrow q},p,q;N,A,B)$. Assume that $M$ and $M'$ differ in rank on $X$.
As $L_{p\rightarrow q}$ is freer than $L$, we have that $M'$ is freer than $M$ and, hence,
$r_{M'}(X)>r_{M}(X)$.
Let $S_1 = E(L)$, $S_2 =A\cup B\cup\{a,b\}$,
$X_1=X\cap S_1$,  
$X_2 = X\cap S_2$, and $\ell=\{p,q\}$.

For $M$ and $M'$ to differ in rank on $X$ it must be the case that $L$ and $L_{p\rightarrow q}$ to differ in rank 
on $X_1\cup \{p\}$. Thus $X_1$ spans $p$ but not $q$ in $L$.

Note that $r_M(X) = r_M(X_1) + r_M(X_2) - \sqcap_M(X_1,X_2)$ and
$r_{M'}(X) = r_{M'}(X_1) + r_{M'}(X_2) - \sqcap_{M'}(X_1,X_2)$, so
$\sqcap_M(X_1,X_2)> \sqcap_{M'}(X_1,X_2)$.
However, $\sqcap_M(X_1,\ell) = \sqcap_{M'}(X_1,\ell)=1$ and
$\sqcap_M(X_2,\ell) = \sqcap_{M'}(X_2,\ell)$. 
Hence $\sqcap_M(X_1,X_2)=1$ and $\sqcap_{M'}(X_1,X_2)=0$.
So $X_2$ spans $p$ in $M$ and, since $M|(S_2\cup \ell)= M'|(S_2\cup \ell)$,
$X_2$ also spans $p$ in $M'$. Since $\sqcap_{M'}(X_1,X_2)=0$, we have that 
$X$ does not span $q$ in $M'$ or in $M$.

Now the result follows from Lemma~\ref{pnotq}.
\end{proof}

We are now ready to prove that proper minors of  $M(L,p,q;N,A,B)$ are $\bF$-representable.
We will, in fact, prove the following more general result.
\begin{theorem}\label{minors}
Let $\cM$ be a natural class of matroids. If
\begin{itemize}
\item[(i)] $L$ is an excluded-minor for $\cM$,
\item[(ii)] $p$ and $q$ are an incomparable pair of elements in $L$ such that $L_{p\rightarrow q}$ and 
$L_{q\rightarrow p}$ are contained in $\cM$, 
\item[(iii)] $N$ is a matroid in $\cM$ with $E(N)\cap E(L)=\emptyset$, and
\item[(iv)] $(A,B)$ is a partition of $E(N)$ into bases,
\end{itemize}
then each proper minor of $M(L,p,q;N,A,B)$ is contained in $\cM$.
\end{theorem}

\begin{proof}
Let $M=M(L,p,q;N,A,B)$.
If $e\in E(L)-\{p,q\}$, then, by construction,
$M\del e=M(L\del e,p,q;N,A,B)$ and $M\con e=M(L\con e,p,q;N,A,B)$.
Then, since $L\del e$, $L\con e$ and $N$ are all contained in 
$\cM$, the minors $M\del e$ and $M\con e$ are also contained in $\cM$.

Now, for $e\in A\cup\{a\}$ and $f\in B\cup\{b\}$, it follows from Lemma~\ref{difference} that
$M\del e=M(L_{p\rightarrow q},p,q;N,A,B)\del e$ and $M\con f=M(L_{p\rightarrow q},p,q;N,A,B)\con f$.
Then, since $L_{p\rightarrow q}$ and $N$ are all contained in 
$\cM$, the minors $M\del e$ and $M\con e$ are also contained in $\cM$.

Finally, since $M(L,q,p;N,B,A)=M(L,p,q;N,A,B)$, it follows that, for $e\in A\cup\{a\}$ and $f\in B\cup\{b\}$,
the minors $M\con e$ and $M\del f$ are contained in $\cM$.
\end{proof}

We have thus proved Theorem~\ref{construction}  --that $M=M(L,p,q;N,A,B)$ is an excluded-minor for $\bF$-representability;
by Lemma~\ref{CONii} we have that $M$ is not $\bF$-representable,
and by Lemma~\ref{minors} that all the minors of $M$ are $\bF$-representable.

We can now prove Theorem~\ref{main}, which we restate here for convenience.
\begin{blanktheorem}[Theorem~\ref{main}]
If $\bF$ is an infinite field and 
$\cM$ is a natural class that properly contains all $\bF$-representable matroids,
then each $\bF$-representable matroid is a minor of an excluded minor for $\bF$-representability
that is contained in $\cM$.
\end{blanktheorem}

\begin{proof}
Let $N_0$ be an $\bF$-representable matroid. By Lemma~\ref{constructed},
there is an $N_0$-constructed matroid $N$, containing $N_0$ as a minor, and a partition 
$(A,B)$ of $E(N)$ into two bases. By Lemma~\ref{special-pair}, there is an excluded minor $L$ for 
$\bF$-representability  such that  $L$ is contained in $\cM$ and  such that $L$ 
contains an incomparable pair $\{p,q\}$ of elements where 
$L_{p\rightarrow q}$ and $L_{q\rightarrow p}$ are both $\bF$-representable.

Let $M=M(L,p,q;N,A,B)$. Note $M$ contains $N$ 
as the minor $M\con\{a,b\}\del E(L)$. By Theorem~\ref{construction}, $M$ is an excluded minor for $\bF$-representability.
Moreover, since $N$ and $L$ are both contained in the natural class $\cM$, the matroid $M$ is also
contained in $\cM.$
\end{proof}

\section{Further uses of the construction}

We do not know the answer to the following question, but
we suspect that the answer is negative for certain choices of $\cM_1$ and $\cM_2$.
\begin{question}
Let $\cM_1$ and $\cM_2$ be natural classes with $\cM_1\subsetneq \cM_2$.
Is it true that each matroid in $\cM_1$ is a minor of an excluded minor for $\cM_1$
that is contained in $\cM_2$?
\end{question}
The construction in the previous section brings us close to a positive answer.

Consider a matroid $N_0$ in $\cM_1$. By Lemma~\ref{constructed},
there is an $N_0$ constructed matroid $N$, containing $N_0$ as a minor, and a partition 
$(A,B)$ of $E(N)$ into two bases. By Lemma~\ref{special-pair}, there is an excluded minor $L$ for 
$\cM_1$  that is contained in $\cM_2$ and  such that $L$ 
contains an incomparable pair $\{p,q\}$ of elements where 
$L_{p\rightarrow q}$ and $L_{q\rightarrow p}$ are both contain in $\cM_1$

Let $M=M(L,p,q;N,A,B)$. Then $M$ contains $N$ 
as the minor $M\con\{a,b\}\del E(L)$. Since $N$ and $L$ are both contained in the natural class $\cM_2$, the 
matroid $M$ is also contained in $\cM_2$.
Moreover, by Theorem~\ref{minors},  each proper minor of $M$ is contained in $\cM_1$.
The only additional property that we require is that $M$ is not contained in $\cM_1$.

Let us review the argument  used to prove that $M\not\in\cM_1$ when $\cM_1$
is the class of matroids represented over an infinite field $\bF$.
Let $S_1 = E(L)-\{p,q\}$, $S_2 = A\cup B\cup\{a,b\}$, $Y_{2p} = A\cup \{a\}$ and
$Y_{2q}=B\cup\{b\}$. By Lemma~\ref{separation}, $(S_1,S_2)$ is a 
$3$-separation in $M$ and $\sqcap_M(S_1,Y_{2p}) = \sqcap_M(S_1,Y_{2q})=1$.
We proceed by way of contradiction, supposing that $M$ is $\bF$-representable.
Thus there is an $\bF$-representable matroid $M'$ obtained by extending
$M$ by two non-loop elements $p$ and $q$ such that $p$ is spanned by both 
$S_1$ and $Y_{2p}$ and $q$ is spanned by both $S_1$ and $Y_{2q}$; 
this is the key step where representability plays its crucial role.

Since $p$ and $q$ are incomparable in $L$, there exist sets $Y_{1p}$ and $Y_{1q}$
such that $Y_{1p}$ spans $p$ but not $q$ whereas $Y_{1q}$ spans $q$ but not $p$.
Note that $\sqcap_M(Y_{1p},S_2)=\sqcap_M(Y_{1p},Y_{2p})=1$ and that
$\sqcap_M(Y_{1q},S_2)=\sqcap_M(Y_{1q},Y_{2q})=1$. By Lemma~\ref{unique},
there is a unique way to extend $M$ by non-loop elements $p$ and $q$
so that $p$ is spanned by both $Y_{1p}$ and $Y_{2p}$ and
$q$ is spanned by both $Y_{1q}$ and $Y_{2q}$. Thus $M' = M_2(L,p,q;N,A,B)$.
However, this contradicts the fact that $M'$ is $\bF$-representable.

To extend this to other choices of $\cM_1$, we require that, if $M\in\cM_1$, then
there exists a matroid $M''$ in $\cM$ obtained by extending
$\cM$ by non-loop elements $p$ and $q$ such that
$p$ is spanned by both $Y_{1p}$ and $Y_{2p}$ and
$q$ is spanned by both $Y_{1q}$ and $Y_{2q}$. We will impose this condition
artificially.

Let $e$ be a non-loop element in a matroid $M$. We say that $M$ is a {\em pinned extension into a $3$-separation}
of $M\del e$ if there is a $3$-separation $(S_1,S_2)$ and sets $Y_1\subseteq S_1$ and $Y_2\subseteq S_2$
such that $\sqcap_{M\del e}(Y_1,S_2) = \sqcap_{M\del e}(S_1,Y_2)=1$ and
$e$ is spanned by both $Y_1$ and $Y_2$. 

The above discussion is summarized in the following result.

\begin{theorem}\label{general}
Let $\cM_1$ and $\cM_2$ be natural classes where $\cM_1\subsetneq \cM_2$.
If $\cM_1$ is closed under pinned extensions into $3$-separations, then
each matroid in $\cM_1$ is a minor of an excluded minor of $\cM_1$ that
is also in $\cM_2$.
\end{theorem}

Note that, by Lemma~\ref{unique}, if $M$ is a pinned extension into a $3$-separation
$(S_1,S_2)$ and $Y_1$ and $Y_2$ are the sets that ``pin" $e$, then 
$M$ is uniquely determined by $M\del e$, $(S_1,S_2)$, $Y_1$, and $Y_2$.
Therefore, the intersection of two classes that are both closed under
pinned extensions into $3$-separations will also be closed under
pinned extensions into $3$-separations.

\section{Acknowledgements}
We thank the referees for their helpful comments. We also thank Thomas Savitsky for pointing out an error in an earlier version of Theorem 3.3 in which we confused gammoids with transversal matroids.

\section*{References}
	\newcounter{refs}
	\begin{list}{[\arabic{refs}]}
		{\usecounter{refs}\setlength{\leftmargin}{10mm}\setlength{\itemsep}{0mm}}
		\item\label{Orientable}
		A. Bj\"{o}rner, M. Las Vergnas, B. Sturmfels, N. White, G. Ziegler,
		Oriented Matroids, Second Edition,
		Cambridge University Press, Cambridge (1999).

		\item\label{Lattice}
		J. Bonin,
		Lattice path matroids: the excluded minors,
		J. Combin. Theory Ser. B {\bf 100} (2010) 585--599.

		\item\label{Brylawski}
		T. Brylawski,
		An affine representation for transversal geometries,
		Studies in Appl. Math. {\bf 54} (1975) 143--160.

		\item\label{Gammoids}
		A. Ingleton, J. Piff,
		Gammoids and transversal matroids,
		J. Combin. Theory Ser. B {\bf 15} (1973), 51--68.

		\item\label{Knuth}
		D. Knuth,
		The asymptotic number of geometries,
		J. Combin. Theory. Ser. A {\bf 16} (1974), 398--400.
		
		\item\label{Lazarson}
		T. Lazarson,
		The representation problem for independence functions,
		J. London Math. Soc. {\bf 33} (1958), 21--25.

		\item\label{Mason}
		J. Mason,
		On the class of matroids arising from paths in graphs,
		Proc. London Math. Soc. {\bf 25} (1972), 55--74.

		\item\label{Matus}
		F. Mat\'{u}\v{s},
		Classes of matroids closed under minors and principal extensions,
		Combinatorica (2017). https://doi.org/10.1007/s00493-016-3534-y

		\item\label{MNW}
		E. Mayhew, M. Newman, G. Whittle,
		On excluded minors for real-representability,
		J. Combin. Theory Ser. B {\bf 99} (2009) 685--689.

		\item\label{Nelson}
		P. Nelson,
		Almost all matroids are nonrepresentable,
		Proc. London Math. Soc. {\bf 50} (2018), 245--248.

		\item\label{O}
		J. Oxley, 
		Matroid Theory,
		Oxford University Press, New York (2011).

		\item\label{Tuttebinary}
		W. Tutte,
		A homotopy theorem for matroids. I, II.
		Trans. Amer. Math. Soc. {\bf 88} (1958), 144--174.
	\end{list}

\end{document}